\documentclass[11pt]{amsart}
\usepackage{amssymb,latexsym}%\documentclass[12pt,leqno]{amsart}
\def\NN{\nu({\cal F}_\xi)}

\def\nt{\nabla^T}
\def\ric{\hbox{Ric}}

\def \cal {\mathcal}

\def\wt{\widetilde}

\def\D{\Delta}
\def\p{\partial}
\def\l{\lambda}

\def\a{\alpha}

\def\cal{\mathcal}

\def\l{\lambda}

\def\be1{{\begin{equation}}}
\def\ee1{{\end{equation}}}

\def\D{\Delta}

\def\p{{\partial}}

\def\l{\lambda}

\def\a{\alpha}

\def\part{\partial}

\def\tr{{\hbox{tr}}}

\def\ba{\begin{array}}
\def\ea{\end{array}}

\newtheorem{corollary}{Corollary}[section]

\numberwithin{equation}{section}

\newtheorem{lemma}{Lemma}[section]
\newtheorem{proposition}[lemma]{Proposition}
\newtheorem{theorem}[lemma]{Theorem}

\title{The Sasaki-Ricci flow on Sasakian 3-spheres}

\author{Guofang Wang}
\address{ Albert-Ludwigs-Universit\"at Freiburg, Mathematisches Institut, Eckerstr. 1, 79104 Freiburg, Germany}
\email{guofang.wang@math.uni-freiburg.de}

\author{Yongbing Zhang}
\address{University of Science and Technology of China, Department of Mathematics, 230026 Hefei, China
\and Albert-Ludwigs-Universit\"at Freiburg, Mathematisches Institut,  Eckerstr. 1, 79104 Freiburg, Germany}
\email{ybzhang@amss.ac.cn}

\subjclass[2010]{53C44, 53C25}
\keywords{Sasaki-Ricci flow, Sasaki-Ricci soliton, weighted Sasakian structure}

\begin{document}
\maketitle

\begin{abstract}
We show that on a Sasakian 3-sphere the Sasaki-Ricci flow initiating from a Sasakian metric of positive transverse scalar curvature
converges to a gradient Sasaki-Ricci soliton.
We also show the existence and uniqueness of gradient Sasaki-Ricci soliton on each Sasakian 3-sphere.
\end{abstract}

\section{Introduction}

Recently  Sasaki-Einstein geometry, as an odd-dimensional cousin of K\"ahler-Einstein geometry,
has played an important role in the Ads/CFT correspondence.
The important problem in Sasaki-Einstein geometry is certainly to find Sasaki-Einstein metrics.
Boyer, Galicki and their collaborators found many new Sasaki-Einstein metrics on quasi-regular Sasakian manifolds
\cite{BG2}, \cite{BGK} and \cite{BGN}. The first class of irregular Sasaki-Einstein metrics was found
by Gauntlett, Martelli, Spark and Waldram in \cite{GMSW} and \cite{GMSW2}.
Another class of irregular Sasaki-Einstein metrics was found in \cite{Wang} by studying the Sasaki-Ricci solitons.

In order to systematically study the existence of the Sasaki-Einstein metrics, we introduced in \cite{SWZ}
a Sasaki-Ricci flow, motivated by the work of  Lovri\'c, Min-Oo and Ruh \cite{Oo} in transverse Riemannian geometry.
The Sasaki-Ricci flow exploits the transverse structure of Sasakian manifolds.
We have showed in \cite{SWZ} the well-posedness of the Sasaki-Ricci flow and global existence of the flow,
together with a Cao \cite{Cao} type result, i.e., the convergence in the case of negative and null basic first Chern class.
For the more precise definitions see Section 2.

In the paper we want to consider the Sasaki-Ricci flow in the three-dimensional case.
Three-dimensional Sasakian manifolds are of the lowest dimension for Sasakian geometry.
Nevertheless, the Sasakian structures on three-dimensional manifolds are quite rich and have been well studied \cite{Belgun00,Belgun01,GO,Geiges}.
A compact 3-manifold admits a Sasakian structure if and only if it is diffeomorphic to one of the following standard models:
(i) $ S^3/\Gamma,$  (ii) $ Nil^3/\Gamma,$ (iii) $ \widetilde{SL(2,\mathbb{R})}/\Gamma,$
where $\Gamma$ is a discrete subgroup of the isometry group with respect to the standard Sasakian metric in each case.
$Nil^3$ is for the 3 nilpotent real matrices (i.e. the Heisenberg group),
and $\widetilde{SL(2,\mathbb{R})}$ is the universal cover of $SL(2,\mathbb{R})$.
Note that cases (i), (ii), and (iii) correspond to Sasakian manifolds
with positive, null, and negative basic first Chern classes, respectively.

As mentioned above that on the Sasakian manifolds with null or negative basic first Chern classes
the Sasaki-Ricci flow converges. Therefore, we  focus in this paper on three-dimensional Sasakian manifolds, i.e., Sasakian 3-spheres.

Let $$S^3=\{z\in \mathbb{C}^2, |z_1|^2+|z_2|^2=1 \}$$
and
$$\eta_0=\sum_{i=1}^2(x^idy^i-y^idx^i).$$
This $\eta_0$ together with the standard almost complex structure gives the standard Sasakian structure
and the corresponding Sasakian manifold is the standard sphere.
A {\it weighted Sasakian structure} on $S^3$ is given by
$$\eta_a=(a_1|z_1|^2+a_2|z_2|^2)^{-1}\eta_0,$$
where $a_1$ and $a_2$ are any positive numbers.
The Reeb vector field of the weighted Sasakian structure is
$$\xi_a=\sum_{i=1}^2a_i(x^i\frac{\partial}{\partial y^i}-y^i\frac{\partial}{\partial x^i}).$$
When $a_1/a_2$ is a rational number, the weighted Sasakian structure is quasi-regular; otherwise it is irregular.
We know that for any Sasakian structure $(S^3, \eta)$ on $S^3$ there is a weighted Sasakian structure such that
\[[d\eta]_B=[d\eta_a]_B,\]
where $[\cdot]_B$ is the basic homology class; for the proof see Proposition 6 in \cite{Belgun01}. See Section 2.
The Sasakian structure $(S^3, \eta)$ also has Reeb vector field $\xi_a$, and we say that it is homologous to $\eta_a$.
The Sasaki-Ricci flow deforms Sasakian structures within a fixed class $[\cdot]_B$.

Our main result is
\begin{theorem}\label{Main1}
For any initial Sasakian structure on $S^3$ of positive transverse scalar curvature,
the Sasaki-Ricci flow converges exponentially to a gradient Sasaki-Ricci soliton.
\end{theorem}

Moreover, we prove the existence and uniqueness of gradient Sasaki-Ricci soliton on any weighted Sasakian structure.

\begin{theorem} \label{Main2}
For any given homologous class of Sasakian structures on $S^3$, there exists a unique gradient Sasaki-Ricci soliton.
\end{theorem}

The proof follows closely Hamilton's idea \cite{H} on the Ricci flow on surfaces
and relies on the Li-Yau-Harnack inequality and the entropy formula.
Actually, the Sasaki-Ricci flow on Sasakian 3-spheres shares lots of
properties of  the Ricci flow on a 2-sphere and two-dimensional orbiflods,
which was studied by Hamilton \cite{H} and Langfang Wu \cite{Wu},
see also the work of B. Chow \cite{Chow} and \cite{ChowWu}.
In the expression of weighted Sasakian structures, if
$a_1/a_2=1$, the characteristic foliation is regular and the leaf space is a two-dimensional sphere.
In this case the Sasaki-Ricci flow is equivalent to Hamilton's Ricci flow on a 2-sphere.
If $a_1/a_2\neq 1$ is a rational number, the leaf space is a so-called bad orbifold with one or two orbifold points.
Our Sasaki-Ricci flow in this case is equivalent to the Ricci flow on such orbifolds studied by L. Wu \cite{Wu}.
The key step in her proof is to establish some injectivity radius estimates and volume estimates.
Note that Wu's injectivity radius estimates are just for the orbifold points
and the points in a region apart from the orbifold points at a distance.
One main reason makes the injectivity radius estimate near orbifold points fail is that near a $p$-fold point,
a shortest geodesic 1-gon may be generated by each of $p$ copies of short geodesic segments on the universal covering.
For the quasi-regular Sasakian structures, the Sasaki-Ricci flow can be reduced to the Ricci flow on bad orbifolds.
Then one can apply Wu's estimates to get the convergence of the Sasaki-Ricci flow to a soliton solution.
The Ricci flow on bad orbifolds with negative curvature somewhere was also studied by Chow-Wu \cite{ChowWu}.

The remaining case of an irrational ratio $a_1/a_2$ is not covered by the
work of Wu \cite{Wu} and Chow-Wu \cite{ChowWu}. In this case the leaf space
has no manifold structure and it could be very wild.
We have to work directly on the three-dimensional manifold $S^3$.
Nevertheless we manage to show the crucial volume estimate by using the Weyl tube formula.

 % However our proof of Theorem \ref{Main1} does not
 %rely on the specific forms of the Sasakian structures. What we
 %really concern are domains which admit local submersion structures
 %but not the leaves space. This allows us to unify the treatments for
 %all Sasakian structures on $S^3$.

The rest of paper is organized as follows. In section 2
we recall the definitions of Sasakian manifolds,  Sasaki-Ricci flow, Sasaki-Ricci solition. In Section 3
we prove the convergence of the Sasaki-Ricci flow to a
gradient Sasaki-Ricci soliton, by using the Harnack inequality and
the entropy formula and leave  the crucial volume estimates in Section 4. In Section 5  we study the gradient
Sasaki-Ricci soliton explicitly.

\section{Sasakian manifolds and Sasaki-Ricci flow}

\subsection{Sasakian manifolds}
For convenience of the reader, we recall briefly the definitions of Sasakian manifolds, its
basic concepts and our Sasaki-Ricci flow. For more details we refer to \cite{Boyer} and \cite{SWZ}.

Let $(M, g^M)$ be a Riemannian manifold, $\nabla^M$ the Levi-Civita
connection of the Riemannian metric $g^M$, and let $R^M(X,Y)$ denote the
Riemann curvature tensor of $\nabla^M$.
%% change7 (Riemann) to (Riemannian)--------------------------------------------------------
By a {\it contact manifold} we mean a $C^\infty$ manifold $M^{2n+1}$
together with a $1$-form $\eta$ such that $\eta \wedge (d\eta)^n\neq 0$.
%In particular $\eta \wedge (d\eta)^n$ defines a volume
%element on $M$, and hence $M$ is orientable.
It is easy to check that there is a canonical vector field $\xi$ defined by
\[\eta (\xi)=1 \quad \text{ and } \quad d\eta (\xi, X)=0,
%% change8  to (\quad \text{ and } \quad)-------------------------------------------------------------
\hbox { for any vector field $X$}. \]
The vector field $\xi$ is called the {\it characteristic vector field} or {\it Reeb vector field}. Let
\begin{equation}\nonumber
{\cal D}_  p=\text{ker}\eta_ p.
\end{equation}
There is a decomposition of the tangential bundle $TM$
\[TM={\cal D} \oplus L_\xi,\]
where $L_\xi$ is the trivial bundle generated by the Reeb field $\xi$.
%A Riemannian metric
%$g=\dd g\alpha\beta dy^\alpha\otimes dy^\beta\in\Gamma(T^*M\otimes T^*M)$
%on a contact manifold
%$(M,\eta)$ is said to be {\it associated}, if
%\begin{equation}\label{n1}
%\uu g\alpha\beta\eta \beta=\xi \eta^\alpha,
%\end{equation}
%i.e.
%\begin{equation}\nonumber
%g(\xi,X)=\eta(X),~~\forall~ X\in TM.
%\end{equation}
%$(
A contact manifold with a Riemannian metric $g^M$ and a tensor field
$\Phi$ of type (1,1) satisfying
\[\Phi^2=-I+\eta\otimes \xi \quad \hbox{
and } \quad
g^M(\Phi X, \Phi Y)=g^M(X, Y)-\eta(X)\eta(Y)\]
is called an {\it almost metric contact manifold}.
Such an almost metric contact manifold is called {\it Sasakian} if one of the
following equivalent conditions holds:
\begin{itemize}
\item[(1)]
There exists a Killing vector field $\xi$ of unit length on $M $ so
that the Riemann curvature satisfies the condition
$$R^M(X,\xi)Y ~=~ g^M(\xi,Y)X-g^M(X,Y)\xi,$$
for any pair of vector fields $X$ and $Y$ on $M$.
\item[(2)] The metric cone $(C(M),{\bar g})= (\mathbb{R}_ +\times M, \ dr^2+r^2g^M)$ is K\"ahler.
\end{itemize}
For other equivalent definitions and the proof of the equivalence,
see for instance \cite{BG1}. By (2), a Sasakian manifold can be viewed as an
odd-dimensional counterpart of a K\"ahler manifold.

A Sasakian manifold $(M, \xi,\eta,\Phi, g^M)$ is a Sasaki-Einstein
%%change9 (Sasakian ... is called a ...)----------------------------------------------
manifold if $g^M$ is an Einstein metric, i.e.,
\[\ric_{g^M}=cg^M,\]
for some constant $c$. Due to property (1) of the Sasaki-Einstein manifold,
it is easy to see that $c=2n>0$. A generalized Sasaki-Einstein metric, $\eta$-Einstein manifold, is defined by
\begin{equation}\label{eta}
\ric_{g^M}=\l g^M+ \nu \eta\otimes \eta,\end{equation}
for some constant $\l$ and $\nu$. It is easy to see that $\l+\nu=2n$.
For a recent study of $\eta$-Einstein manifolds and  Sasaki-Einstein metrics, see \cite{BGM}, \cite{MSY} and \cite{Wang}.

\subsection{Transverse K\"ahler structures}

In order to study the analytic aspect of Sasaki-Einstein metrics or $\eta$-Einstein manifolds,
we need to consider the transverse structure of Sasakian manifolds.
In this paper, we always assume that $M$ is a Sasakian
%%change11 to (Sasakian)------------------------------------------------------------
manifold with Sasakian Structure $(\xi, \eta, g^M, \Phi)$.
%To well
%understand the Sasakian structure, we consider the {\it transverse
%K\"ahler} structure on $M$.
Let ${\cal F}_  \xi$ be the {\it characteristic foliation} generated by $\xi$.
%The contact
%subbundle ${\cal D}$ is just the normal bundle to the characteristic
%foliation ${\cal F}  \xi$.
On ${\cal D}$, it is naturally endowed with both a complex structure
$\Phi_  {|{\cal D}}$ and a symplectic structure $d\eta$. $({\cal D},
\Phi_  {|{\cal D}}, d\eta, g^T)$ gives $M$ a transverse K\"ahler
structure with K\"ahler form $d\eta$ and transverse metric $g^T$
defined by
\[g^T(X,Y)=d\eta(X, \Phi Y).\]
The metric $g^T$ is clearly related to the Sasakian metric $g^M$ by
\[g^M=g^T+ \eta\otimes \eta.\]

There is a canonical  quotient bundle of the foliation ${\cal F}_  \xi$,
$\NN=TM/{L_\xi}$ and an isomorphism between $\NN$ and ${\cal
D}$. Let $p:TM\to \NN$ be the projection. $g^T$ gives a bundle map
$\sigma :\NN \to \mathcal{D}$ which splits the exact sequence \[0\to L_\xi \to
TM\to \NN\to 0,\]  i.e. $p \circ \sigma =id$.

%Denoting $\sigma^*g^T$
%still by $g^T$,  $\sigma: (\NN, g^T)\to ({\cal D}, g^T)$ is a metric
%isomorphism.  We might identify ${\cal D}$ with $\NN$. However,
%since we want to  deform Sasakian metrics, it is better to
%distinguish them.  In fact, under the deformations considered later,
%$\cal D$ changes,  while $\NN$ keeps fixed. But,  for simplicity of
%notation,  we will use the same notation and $\sigma=id$ if there is
%no confusion,  especially if we do not consider deformations.

From the transverse metric $g^T$, one
can define a transverse Levi-Civita connection on $\NN$ by
 \begin{equation}\label{conn}
 \nabla^T_  X V=\left\{ \begin{array}{ll} (\nabla^M_   X \sigma (V) ) ^p,
& \hbox{ if } X \hbox{ is a section of } {\cal D},\\
{[\xi, \sigma (V) ]}
 ^p, & \hbox{ if } X=\xi, \end{array}\right.
\end{equation}
where $V$ is a section of $\NN$ and $X^p=p(X) $ the projection of
$X$ onto $\NN$ and $\nabla^M$ is the Levi-Civita connection associated to the
Riemannian metric $g^M$ on $M$.
The  transverse
curvature operator is defined  by
\[R^T(X,Y)=\nt_  X \nt_  Y-\nt_  Y\nt_  X-\nt_  {[X,Y]}\]
and
transverse Ricci curvature by
\[\ric^T(X,Y)=g^M(R^T(X,e_  i) e_  i, Y),\]
where $e_ i$ is an orthonormal basis of ${\cal D}$. We remark that
here we have used the identification between $\cal D$ and $\NN$.
More precisely the transverse Ricci tensor is defined by
\[\ric^T(X,Y )=g^M(R^T( X,e_  i) \sigma^{-1} (e_  i), Y)\]
for $X, Y \in TM$.
One can check  that
\begin{equation}\label{ric}
\ric^T(X,Y)=\ric^M(X,Y)+2g^T(X,Y).\end{equation}

A  transverse Einstein metric $g^T$ is a
transverse metric satisfying
\begin{equation}\label{tEinstein}
\ric^T=cg^T,\end{equation} for certain constant $c$.
It is clear that a Sasakian metric is a transverse Einstein metric if and
only if it is an $\eta$-Einstein metric.

In order to introduce the Sasaki-Ricci flow, we first consider
deformations of Sasakian structures which preserve the Reeb field
$\xi$, and hence the characteristic foliation ${\cal F}_  \xi$.

A $p$-form $\a$ on $M$ is called {\it basic} if it satisfies
%%change 12 to (if it satisfies ...)-----------------------------------------------
$i(\xi)\a=0,\quad {\cal L}_  \xi \a=0.$
A function $f$ is basic if and only if $\xi(f)=0$.
One can check that the exterior differential $d$ preserves basic
forms. Hence one can define the
basic cohomology  in a usual way. See \cite{Boyer}.
Moreover, we consider the complexified bundle
 ${\cal D}^{\mathbb C}={\cal D}\otimes {\mathbb C}$.
 Using the structure $\Phi$ we decompose ${\cal D}^{\mathbb C}$ into
 two subbundles ${\cal D}^{1,0}$
 and ${\cal D}^{0,1}$, where ${\cal D}^{1,0}=\{X\in
 {\cal D}^{\mathbb C}\,|\, \Phi X=\sqrt{-1}X\}$
 and ${\cal D}^{0,1}=\{X\in
 {\cal D}^{\mathbb C}\,|\, \Phi X=-\sqrt{-1}X\}$.
 Similarly, we decompose the complexified space
 $\Lambda_  B^r\otimes {\mathbb C}=
 \oplus_   {p+q=r}\Lambda_   B^{p,q}$, where $\Lambda_ B^{p,q}$
 denotes the sheaf of germs of basic forms of type $(p,q)$.
Define $\partial_  B$ and $\bar \partial_  B$ by
\[
\p_  B:\Lambda_  B^{p,q} \to   \Lambda_  B^{p+1,q}, \quad
\bar \p_  B:\Lambda_  B^{p,q}\to  \Lambda_
B^{p,q+1},\] which is the decomposition of $d$. Let $d_
B=d_ {| {\Omega^p_  B}}$. We have $d_  B=\p_  B+\bar \p_  B$. Let
$d^c_ B=\frac 12 {\sqrt{-1}} (\bar \p_ B-\p_  B).$
Let $d^*_  B:\Omega^{p+1}_  B\to \Omega^{p}_  B$ be the adjoint
operator of $d_  B:\Omega^p_  B\to \Omega^{p+1}_  B$. The basic
Laplacian $\D_  B$ is defined
\[\D_  B=d^*_  Bd_  B+d_  Bd_  B^*.\]

Suppose that $(\xi,\eta,\Phi,g^M)$ defines a Sasakian structure on
$M$. Let $\varphi$ be a basic function. Put
\[\tilde \eta=\eta+d^c_  B\varphi.\]
It is clear that
\[d\tilde \eta=d\eta+d_  B d_  B^c\varphi=d\eta+\sqrt{-1}\p_  B\bar \p_  B \varphi.\]
For small $\varphi$, $d\tilde\eta$ is non-degenerate in the sense
that $\tilde\eta\wedge (d\tilde\eta)^n\neq 0$.  Set
\[
\wt\Phi= \Phi-\xi\otimes (d^c_  B\varphi)\circ \Phi, \quad
\tilde{g}^M=d\wt\eta \circ (Id\otimes \tilde \Phi)+\tilde \eta
\otimes \tilde \eta.\]
 $(M, \xi,\tilde \eta, \wt \Phi, \tilde{g}^M)$
is also a Sasakian structure with $[d\tilde\eta]_B=[d\eta]_B$. It is this class of deformations we used in the definition of
our Sasaki-Ricci flow.

There are other kinds of deformation. For instance, the so-called
${\cal D}$-homothetic deformation is defined
\[\bar \eta =a \eta, \quad \bar \xi=\frac 1 a \xi,\quad \bar \Phi =\Phi,
\quad \bar{g}^M= ag^M+a(a-1)\eta\otimes \eta\] for a positive constant
$a$. Note  that from an $\eta$-Einstein metric with $\l>-2$, one can
use the ${\cal D}$-homothetic deformation to get a Sasaki-Einstein
metric.  It was called also $0$-type deformation.

A first type
deformation of this Sasakian structure is a new Sasakian structure
$(M, \eta', \mathcal{D}, \Phi', \xi')$, where $\eta'=f\eta$,
for a positive function $f\neq \text{constant}$, and $\xi'$ is
the corresponding Reeb vector field, where $\Phi'|_{\mathcal{D}}=\Phi|_{\mathcal{D}}$.
See for example, \cite{Belgun00}, \cite{Belgun01} and \cite{GO}. In this terminology our deformation is called
the second type deformations. Here we would like to call them canonical deformations.
A second type deformation $\eta'$ of $\eta$ is also called homologous to $\eta$ \cite{BGN}.

\subsection{Sasaki-Ricci flow}
 Let $\rho^T=\ric^T(\Phi\cdot,\cdot)$ and $\rho^M=\ric^M(\Phi\cdot,\cdot)$.
 $\rho^T$ is called the {\it transverse Ricci form}. One can check that
 in view of (\ref{ric}) we have
 \begin{equation}\label{ric2}
 \rho^T=\rho^M+2d\eta.\end{equation}
 $\rho^T$ is a closed basic form and its
 basic cohomology class
 $[\rho^T]_  B=c^1_  B$ is the basic first Chern class.
 $c^1_  B$ is called {\it positive} ({\it negative, null} resp.) if it contains a
 positive (negative, null resp.) representation.
 The transverse Einstein equation (\ref{tEinstein}) can be written as
 \begin{equation}\label{tEinstein2}
 \rho^T=c d\eta,\end{equation}
 for some constant $c$. %Since we want to consider equation (\ref{tEinstein2}),
A necessary condition for the existence of (\ref{tEinstein2}) is
\[c_  B^1=c [d\eta]_  B.\]
By a ${\cal D}$-homothetic deformation, it is natural to consider
 \begin{equation}\label{tt} c_  B^1=\kappa [d\eta]_  B,\end{equation}
where $\kappa=1, -1, 0$ corresponds to positive, negative and null
$c^1_  B$.

Now we  consider the following flow $(\xi, \eta(t),\Phi(t),g^M(t))$
with initial data $(\xi, \eta(0),\Phi(0),g^M(0))=(\xi,\eta,\Phi,g^M)$
\begin{equation}\label{flow1}
\frac{d}{dt}g^T(t)=-(Ric^T_  {g^M(t)}-\kappa g^T(t)),\end{equation} or
equivalently
\begin{equation}\label{flow2}
\frac{d}{dt}d\eta(t)=-(\rho^T_  {g^M(t)}-\kappa d\eta(t)).
\end{equation}
We call (\ref{flow1}) {\it Sasaki-Ricci flow.}
In  local coordinates the Sasaki-Ricci flow has the following form
\begin{equation}\label{flow3}
\frac d{dt} \varphi =\log
 \det (g^T_  {i\bar j}+\varphi_  {i\bar j})
 -\log \det (g^T_  {i\bar j})+\kappa \varphi-F,\end{equation}
where the function $F$ is a basic function
such that
\begin{equation}\label{eq4.0}
\rho^T_{g^M}-\kappa d\eta= d_  Bd^c_  B F.
\end{equation}

We showed in \cite{SWZ} that the well-posedness of the Sasaki-Ricci flow and a Cao type result, namely,
if $\kappa =-1$ or $0$, then the Sasaki-Ricci flow converges to an $\eta$-Einstein metric. Unfortunately, in this case,
there is no Sasaki-Einstein metric. Remark that one can obtain in the case $\kappa =-1$ a Sasaki-Einstein Lorentzian metric.
The case $\kappa=1$ is difficult. In general one can only expect to obtain a soliton type solution, namely a Sasaki-Ricci soliton.
A Sasakian structure $(M,\xi,\eta,g^M,\Phi)$ is called a Sasaki-Ricci soliton if there is an Hamiltonian holomorphic vector field $X$ with
\[\rho^T-d\eta={\cal L}_X(d\eta),\]
where ${\cal L}_X$ is the Lie derivative. For the definition of Hamiltonian holomorphic vector field and the study of Sasaki-Ricci solitons
on toric Sasakian manifolds
we refer to \cite{Wang}.

A Sasakian manifold $(M,\xi,\eta,g^M,\Phi)$ is called
 {\it quasi-regular} if there is a positive integer $k$ such that each point has a
foliated coordinate chart $(U; x)$ such that each leaf of $\mathcal{F}_\xi$ passes through $U$ at most $k$
times, otherwise  {\it irregular}. If $k = 1$ then the Sasakian manifold is called regular.
Let $\cal B$ the leave space of $\cal F_\xi$. Then if $M$ is  regular if and only if $\cal B$ is a
K\"ahler manifold and  $M$ is quasi-regular if and only if $\cal B$ is a
K\"ahler orbifold. In these both cases, the Sasaki-Ricci flow  on $M$ is equivalent to
the K\"ahler-Ricci flow on $\cal B$.
But when the Sasakian manifold is  irregular, then  $\cal B$ has no manifold structure.

% {Locally}, if we write $d\eta=\sqrt{-1}g^T {i\bar j}dz^i\wedge
%d \bar z^j$, we can  check that
%$\rho^T=-\sqrt{-1}\partial \bar \partial \log \det
%(g^T {k\bar l })$. Let $\eta(t)=\eta+d B^c\varphi$, a family of basic functions
%$\varphi(t)$
%Then the  flow can be written as
%\begin{equation}\label{flow3}
%\frac d{dt} \varphi =\log
% \det (g {i\bar j}^T+\varphi {i\bar j})
 %-\log (\det g {i\bar j}^T)+\kappa \varphi-H,\end{equation}
% where $h$ is a basic function defined by
% \begin{equation}\label {el} \rho^T g-d\eta=2dd B^c h. \end{equation}
% The solvability of (\ref{el}) was proved in \cite{El}.

%\newpage

%\section{Overview of Hamilton's approach}

\section{Sasaki-Ricci flow on  Sasakian 3-manioolds}

Now let $(M, \xi, \eta, g^T, \Phi)$ be a closed three-dimensional Sasakian manifold.
For simplicity if there is no confusion we remove
superscription $T$, since all quantities we are considering are transverse.
Note that in three-dimensional case $\Lambda_B^{1,1}$ is a line bundle.
Since both the transverse Ricci form $\rho_\eta$ and the transverse K\"ahler form
$d\eta$ are real sections in $\Lambda_B^{1,1}$, the basic first Chern class must be positive, negative, or null.
Actually, let $$\int_M\rho_\eta\wedge\eta=\kappa
\int_Md\eta\wedge\eta,$$
we must have $$[\rho_\eta]_B=\kappa [d\eta]_B.$$

Let $(x, z=x^1+ix^2)$ be the CR coordinates and
$$g_{ij}=d\eta(\frac{\partial}{\partial x^i},\Phi\frac{\partial}{\partial x^j}).$$
Note that
$R_{ij}=\frac{1}{2}Rg_{ij},$
so we can rewrite the Sasaki-Ricci flow as
\begin{equation}\label{SR}\frac{d}{dt}g_{ij}=(r-R)g_{ij},
\end{equation}
where $r$ is the average of the transverse scalar curvature.
For the reason mentioned in the Introduction,
we now focus on the Sasaki-Ricci flow (\ref{SR}) defined on a Sasakian 3-sphere.
Moreover we assume the flow (\ref{SR}) initiating from a metric of positive
transverse scalar curvature.
Note that the transverse scalar curvature $R$ for a Sasakian
manifold is a basic function.

\begin{proposition} \label{fundamentalevolution}
Along the flow (\ref{SR}), we have

$$\frac{d}{dt}\int_Md\eta\wedge\eta=\int_{M}(r-R)d\eta\wedge\eta=0;$$

\begin{equation}\label{evolveR}
\frac{d}{dt}R=\triangle_BR+R(R-r);
\end{equation}

$$\frac{d}{dt}\int_MRd\eta\wedge\eta=0.$$
\end{proposition}

We now assume $R(0)>0$. Note that on Sasakian $3$-spheres, the scalar curvature $R^M=R-2$.
It follows from Proposition \ref{fundamentalevolution} that along the flow (\ref{SR}),
the average transverse scalar curvature $r$ stays as the same constant.
If necessary we can make a $0$-type deformation of the initial Sasakian structure, and always assume that
\begin{equation}\label{Rmax}
R_{\max}:=\max_{x\in S^3}R(x)\geq r\geq 8.
\end{equation}

It follows from the maximum principle that $R(t)>0.$
We follow Hamilton's approach \cite{H} to prove the Harnack inequality and the entropy formula for the flow (\ref{SR}).

\begin{theorem}
Suppose the flow (\ref{SR}) have a solution  for $t<T^*(\leq +\infty)$ with $R(0)>0$.
Then for any two space-time points $(x,\tau)$ and $(y,T)$ with $0<\tau<T<T^*$, we have
\begin{equation}\label{Harnack} (e^{r\tau}-1)R(x,\tau)\leq e^{\frac{1}{4}D}(e^{rT}-1)R(y,T),
\end{equation}
where
\begin{equation}\label{D}
D=D((x,\tau),(y,T)):=\inf_{\gamma}\int_\tau^T|\dot{\gamma}(t)|^2_{g_t}dt.
\end{equation}
Here the infimum is taking over all piece-wisely smooth curves $\gamma(t)$,  $t\in[\tau,T]$, with $\gamma(\tau)=x$ and
$\gamma(T)=y$.
\end{theorem}

\begin{proof}
Let $\gamma(t)$ be any piece-wisely smooth path joining $x$ and $y$, and
$$L=\log R.$$
By the fact that $R$ is a basic function, we have
$$g_t^M(\nabla^ML,\dot{\gamma})=g_t^M(\nabla L,\dot{\gamma})=g_t(\nabla L,\dot{\gamma}).$$
Hence
$$\frac{d}{dt}L(t,\gamma(t))=\frac{\partial L}{\partial t}+g_t(\nabla L,\dot{\gamma})
\geq \frac{\partial L}{\partial t}-|\nabla L|_{g_t}^2-\frac{1}{4}|\dot{\gamma}|^2_{g_t}.$$
It follows from (\ref{evolveR}) that
$$\frac{\partial L}{\partial t}=\triangle_B L+|\nabla L|_{g_t}^2+(R-r).$$
Denote
$$Q=\frac{\partial L}{\partial t}-|\nabla L|_{g_t}^2=\triangle_B L+(R-r).$$
One can compute that
\begin{eqnarray*}
\frac{\partial Q}{\partial t}=\triangle_B Q+2g_t(\nabla L,\nabla Q)+2|\nabla^2L+\frac{1}{2}(R-r)g|_{g_t}^2+r Q.
\end{eqnarray*}
So we have
\begin{eqnarray*}
\frac{\partial Q}{\partial t}\geq \triangle_B Q+2g_t(\nabla L,\nabla Q)+Q^2+r Q,
\end{eqnarray*}
which implies that
$$Q\geq \frac{-re^{rt}}{e^{rt}-1}.$$
Therefore
$$\frac{d}{dt}L(t,\gamma(t))\geq \frac{-re^{rt}}{e^{rt}-1}-\frac{1}{4}|\dot{\gamma}|^2_{g_t}.$$
Taking $\gamma(t)$ to be a path achieving the minima $D$, we get
$$L(T,y)-L(\tau,x)\geq \int_\tau^T(\frac{-re^{rt}}{e^{rt}-1}-\frac{1}{4}|\dot{\gamma}|^2_{g_t})dt
=-\log\frac{e^{rT}-1}{e^{r\tau}-1}-\frac{1}{4}D.$$
\end{proof}

The following entropy formula for the flow (\ref{SR}) is an analog of Hamilton's entropy formula \cite{H}.
For the explicit computation we refer to \cite{Ye}.

\begin{theorem}
Along the flow (\ref{SR}) with $R(0)>0$, the integral $\int_M R\log Rd\mu$ is non-increasing, where $d\mu=d\eta\wedge \eta$.
\end{theorem}

\begin{proof}
Let $f$ be the basic function defined by
$$\triangle_Bf=R-r,$$
and $$M_{ij}=\nabla_i\nabla_jf-\frac{1}{2}(R-r)g_{ij}.$$
Then we have
\begin{eqnarray*}\frac{d}{dt}\int_M R\log Rd\mu&=&\int_M(R-r)^2d\mu-\int_M\frac{|\nabla R|_{g_t}^2}{R}d\mu
\\&=&-\int_M\frac{|\nabla R+R\nabla f|_{g_t}^2}{R}d\mu-2\int_M |M_{ij}|_{g_t}^2d\mu.
\end{eqnarray*}
\end{proof}

In order to combine the Harnack inequality (\ref{Harnack}) and the entropy
formula to obtain uniform upper and lower bounds of $R$, we shall need a crucial volume estimate,
see Lemma \ref{volest} below. To make the statement of Lemma \ref{volest}, we first introduce some notations.

For any two points $x, y\in S^3$ and any transverse metric $g$,  we define the transverse distance by
\begin{equation}\label{transversedistance}d_{g}(x,y)=\inf_{\gamma}\int_{\gamma}|\frac{d}{ds}\gamma(s)|_{g}ds,\end{equation}
here $\gamma(s)$ is any piece-wisely smooth curve joining $x$ to $y$.
For any $p\in S^3$, we denote
$$V_p(d,g)=\{q\in S^3|d_g(p,q)\leq d\}.$$

\begin{lemma}\label{volest}
Let $(S^3,\eta,g)$ be a Sasakian 3-sphere with transverse metric $g$ and positive transverse scalar curvature.
Then there exists a positive constant $C_0$, independent of second deformations of $(\eta,g)$,
such that for any point $p\in S^3$, we have
\begin{equation}\label{volestformula}
\text{Vol}_{g^M}(V_p(\frac{\pi}{2\sqrt{R_{\max}}},g))\geq \frac{C_0}{R_{\max}}.
\end{equation}
\end{lemma}

\begin{proof}
We will prove this Lemma in next section by using Weyl's tube formula.
We treat quasi-regular and irregular Sasakian 3-spheres separately.
For quasi-regular Sasakian 3-spheres, we use essentially Wu's injectivity radius estimate \cite{Wu}, see Proposition \ref{prop1}.
For the case of irregular Sasakian 3-spheres, see Proposition \ref{prop2}.
\end{proof}

\begin{theorem}
Along the flow (\ref{SR}) with $R(0)>0$, there exist constants $c_1>0$ and $c_2<\infty$ such that
$$c_1\leq R(t)\leq c_2.$$
\end{theorem}

\begin{proof}
Without loss of generality, let
$$\tau\geq 1, \quad T= \tau+\frac{1}{2R_{\max}(\tau)}.$$
It follows from (\ref{evolveR}) that $R_{\max}(T)\leq 2R_{\max}(\tau)$.
Then by the definition of the flow (\ref{SR}), $g_t, t\in [\tau,T]$, are equivalent, i.e.
there exist uniform positive constants $\delta_1, \delta_2$ such that
$$\delta_1g(T)\leq g(t)\leq \delta_2 g(T).$$
Let $\gamma(t)$ be the shortest curve ,joining $x$ and $y$, with respect to $g(T)$.
By the definition (\ref{D}) of $D$, we have
\begin{equation}\label{Destimate1}
D((x,\tau),(y,T))\leq \delta_2\int_\tau^T|\dot{\gamma}|^2_{g(T)}dt=\delta_2\frac{d_{g(T)}(x,y)^2}{T-\tau}.
\end{equation}
Taking $x=y=p$ be a point where $R_{\max}(\tau)$ is achieved, it follow from the Harnack inequality (\ref{Harnack}) that
\begin{equation}\label{Restimate1}
R_{\max}(T)\geq \frac{e^{r\tau}-1}{e^{rT}-1}R_{\max}(\tau).
\end{equation}
We shall apply the Lemma \ref{volest} with
$$g=g(T), \quad d=\frac{\pi}{2\sqrt{R_{\max}(T)}}.$$

It follows from (\ref{Destimate1}), (\ref{Restimate1}) that there exists a positive constant $C(r)$ such that
$$D((p,\tau),(q,T))\leq C(r)\delta_2, \quad \text{for} \quad q\in V_p(\frac{\pi}{2\sqrt{R_{\max}(T)}},g(T)).$$
By the Harnack inequality (\ref{Harnack}) again, there exists a positive constant $C(r,\delta_2)$ such that
$$R(q,T)\geq C(r, \delta_2) R_{\max}(\tau), \quad \text{for} \quad q\in V_{p}(\frac{\pi}{2\sqrt{R_{\max}(T)}},g(T)).$$
Note that $$R\log R\geq -\frac{1}{e}.$$
Integrating over $S^3$ at time $T$, we get
\begin{eqnarray*}\int_{S^3}(R\log R+\frac{1}{e})d\mu_{g^M(T)}&\geq& \int_{V_p(\frac{\pi}{2\sqrt{R_{\max}(T)}},g(T))}R\log Rd\mu_{g^M(T)}
\\&\geq& C(r, \delta_2)R_{\max}(\tau)\log (C(r,\delta_2)R_{\max}(\tau))\frac{C_0}{R_{\max}(T)}
\\&\geq& \frac{1}{2}C_0C(r, \delta_2)\log (C(r,\delta_2)R_{\max}(\tau))
.\end{eqnarray*}
Now by the entropy formula, we see that there exists a positive constant $C_1$ such that
$$R_{\max}(\tau)\leq C_1.$$

One can use the upper bound of $R$ to conclude a lower bound of $R$.
Since the volume is constant and $R$ has a upper bound, by Lemma
\ref{volest} we see that the transverse diameter must be bounded from above.
Otherwise, we would have two much volume.
Let $x_0$ be a point with $R(x_0,\tau)\geq r$.
For $\tau\geq 1$, $g_t, t\in [\tau,\tau+1]$, are equivalent.
Hence it follows from the upper bound of the transverse diameter that, there exists a positive constant $C$ such that
$$D((\tau,x_0), (\tau+1,x))\leq C, \quad \forall x.$$
It follows from the Harnack inequality that $R(\tau+1)$ has a uniform lower bound.
\end{proof}

We now show that the flow (\ref{SR}) converges to a gradient
Sasaki-Ricci soliton. For a Sasakian 3-sphere $(S^3,\eta)$, let $f$ be the basic
function defined by $$\triangle_B f=R-r,$$ and
$$M_{ij}=\nabla_i\nabla_jf-\frac{1}{2}\triangle_B fg_{ij}.$$
Then $(S^3,\eta)$ is a gradient Sasaki-Ricci soliton if and only if
$$M_{ij}=0.$$
We will show that the Sasaki-Ricci flow (\ref{SR}) converges to a
solution satisfying $M_{ij}=0$, which means the Sasaki-Ricci soliton
solution is generated by diffeomorphisms related to $X=-\frac{1}{2}\nabla f$.

\begin{theorem}
Along the Sasaki-Ricci flow (\ref{SR}) with $R(0)>0$, we have
$$|M_{ij}|^2\leq Ce^{-ct}.$$
Hence $M_{ij}$ converges to zero exponentially.
\end{theorem}
\begin{proof} It is a direct consequence of the evolution equation of $|M_{ij}|^2$:
$$\frac{\partial}{\partial t}|M_{ij}|^2=\triangle_B |M_{ij}|^2-2|\nabla_kM_{ij}|^2-2R|M_{ij}|^2.$$
\end{proof}

So we have proved the following
\begin{theorem}\label{Main} The Sasaki-Ricci flow (\ref{SR}) on a Sasakian 3-sphere with $R(0)>0$ converges exponentially
to a gradient Sasaki-Ricci soliton.
\end{theorem}

\section{Volume estimates}

In this section, we show Lemma \ref{volest}. We use a Weyl type tube formula to prove the volume estimate (\ref{volestformula}).
We treat quasi-regular Sasakian 3-spheres and irregular Sasakian structures separately. We first recall Weyl's tube formula.

Let $P^q$ be a $q$-dimensional embedded closed submanifold in $M^n$.
A tube $T(P,r)$ of radius $r\geq 0$ about $P$ is the set
\begin{eqnarray*}
T(P,r)=\{x\in M| \quad dist_{g^M}(x,P)\leq r\}.
\end{eqnarray*}
A hypersurface of the form
$$P_t=\{x\in T(P,r)| \quad dist_{g^M}(x,P)=t\}$$
is called the tubular hypersurface at a distance $t$ from $P$.
Let $A_P^M(t)$ denote the $(n-1)$-dimensional area of $P_r$, and $V_P^M(r)$ denote the $n$-dimensional volume of $T(P,r)$.

Let $\nu$ denote the normal bundle of $P$, and $\exp_\nu$ be the exponential map. Then we define
$\text{minfoc}(P)$ to be the supremum of $r$ such that
\begin{eqnarray*}\label{minfoc}
\exp_\nu: \{(p,v)\in\nu, |v|\leq r \}\rightarrow T(P,r)
\end{eqnarray*}
is a diffeomorphism.
For the exponential map $\exp_\nu$, we would like to note the following well-known fact:
any geodesic $\gamma(t)$ in a Sasakian $(S^3,g^M)$
with $\dot{\gamma}(0)\in \mathcal{D}_{\gamma(0)}$ must be horizontal.

\begin{proposition} \label{radialhor} Let $\gamma(t)$ be a geodesic in $(S^3,g^M)$ with
$$\gamma(0)=p, \quad \dot{\gamma}(0)\in \mathcal{D}_p.$$
Then we have
$$\dot{\gamma}\in \mathcal{D}_{\gamma(t)}.$$
\end{proposition}

\begin{proof}
Let $p$ be a given point in $S^3$ and $\gamma(t)$ be a geodesic of arc-lengthly parameterized through $p$ in $(S^3,g^M)$. Set
$$\dot{\gamma}=H+V,$$
where $H$ is the horizontal part and V is the vertical part. Then we have
$$\frac{d}{dt}g^M(\dot{\gamma},\xi)=g^M(\dot{\gamma},\nabla^M_{\dot{\gamma}}\xi)=g^M(\dot{\gamma},\Phi H)=g^M(H,\Phi H)=0.$$
\end{proof}

Let $S^{n-q-1}(\nu_p)$ denote the unit sphere in $\nu_p$ and $d\mu_\nu$ be the standard volume $n$-form on the normal bundle $\nu$.
One can introduce the following function on the normal bundle
$$\Theta_u(p,t)=\frac{\exp_\nu^*[d\mu_M(\exp_\nu(p,tu))]}{d\mu_\nu(p,tu)}$$
for $$0\leq t<\text{minfoc}(P), \quad u\in S^{n-q-1}(\nu_p).$$
For $0< t,r <\text{minfoc}(P)$, we have
\begin{equation}\label{APM}
A_P^M(t)=t^{n-q-1}\int_P\int_{S^{n-q-1}(\nu_p)}\Theta_u(p,t)du dp,
\end{equation}
and
\begin{equation}\label{VPM}
V_P^M(r)=\int_0^rA_P^M(t) dt.
\end{equation}

Let $\sigma(x)$ be the distance function from $P$ to $x$ and
$$N=\nabla^M\sigma$$
be the formal outward unit normal vector field of $P_{\sigma(x)}$. On the tubular hypersurfaces $P_t$, let
$$Su=\nabla^M_uN,$$
here $u\in TP_t$, be the shape operator. The shape operator $S$ satisfies the Riccati differential equation
\begin{equation}\label{Riccati} S'(t)=-S(t)^2+R^M_N,
\end{equation}
here $$S'=\nabla^M_NS, \quad R^M_Nu=R^M(N,u)N, \quad u\in TP_t.$$
We also have
\begin{equation}\label{changeofareaelement}\Theta_u'(t)=(\tr S(t)-\frac{n-q-1}{t})\Theta_u(t),
\end{equation}
here we omit the parameter $p\in P$. Now we can introduce a Weyl type tube formula which we will use in this section,
see for instance \cite{Gray}.

\begin{theorem} Suppose that $P\subset M^n$ is a $q$-dimensional closed submanifold and the sectional curvature of $M$ satisfies
$$K^M\leq \lambda.$$
Then for $0<r<\text{minfoc}(P)$, we have
\begin{eqnarray}\label{Weyl}
A_P^M(r)\geq \frac{2\pi^{\frac{n-q}{2}}}{\Gamma(\frac{n-q}{2})}
&&\sum_{c=0}^{[\frac{q}{2}]}\frac{k_{2c}(R^P-R^M)}{(n-q)(n-q+2)\cdot\cdot\cdot(n-q+2c-2)}
\\&&\cdot(\cos(r\sqrt{\lambda}))^{q-2c}(\frac{\sin(r\sqrt{\lambda})}{\sqrt{\lambda}})^{n-q+2c-1},\nonumber
\end{eqnarray}
here $$k_0(R^P-R^M)=\int_Pd\mu_P.$$
\end{theorem}

\begin{corollary}
Let $P$ be a closed embedding curve $P$ in $(S^3, g^M)$ and
$$K^{M}\leq \lambda.$$
Then for $0<r<\text{minfoc}(P)$, we have
\begin{equation}\label{Weyl2}
A_P^M(r)\geq \pi L(P)\frac{\sin(2r\sqrt{\lambda})}{\sqrt{\lambda}}.
\end{equation}
\end{corollary}

In the sequel we let
$$\lambda=K_{\max}^M.$$
Note that
$$R_{\max}=R^M_{\max}+2=2K^M_{\max}(\mathcal{D})+6.$$
Then the assumption (\ref{Rmax}) is equivalent to
$$K^M_{\max}(\mathcal{D})\geq 1.$$

The proof of Lemma \ref{volest} relies on a classification result \cite{Belgun01,GO} of all Sasakian structures on $S^3$.
Let $$S^3=\{z\in \mathbb{C}^2, |z_1|^2+|z_2|^2=1 \}$$
and
$$\eta_0=\sum_{i=1}^2(x^idy^i-y^idx^i).$$
This $\eta_0$ together with the standard almost complex structure give the standard Sasakian structure
and the corresponding Sasakian manifold is the standard sphere.
A weighted Sasakian structure on $S^3$ is given by
$$\eta_a=(a_1|z_1|^2+a_2|z_2|^2)^{-1}\eta_0,$$
here $a_1$ and $a_2$ are any positive numbers.
The Reeb vector field of the weighted Sasakian structure is
$$\xi_a=\sum_{i=1}^2a_i(x^i\frac{\partial}{\partial y^i}-y^i\frac{\partial}{\partial x^i}),$$
which is generated by the action
$$e^{it}.(z_1,z_2)=(e^{ia_1t}z_1,e^{ia_2t}z_2), \quad t\in [0,2\pi].$$
When $a_1/a_2$ is a rational number, the weighted Sasakian structure is quasi-regular; otherwise it is irregular.
Belgun's classification (see Proposition $6$ in \cite{Belgun01}) tells that any Sasakian structure on $S^3$ is
a second deformation of some weighted Sasakian structure on $S^3$.

We will use formulas (\ref{VPM}) and (\ref{Weyl2}) to prove the volume estimate (\ref{volestformula}).
We first deal with the case of quasi-regular Sasakian structures.
The proof relies essentially on Wu's injectivity radius estimates \cite{Wu}.
In fact in the case of quasi-regular Sasakian structures, the volume estimates can be deduced directly from Wu's volume estimate.

\begin{proposition}\label{prop1}
Let $(S^3,\eta,g)$ be a quasi-regular Sasakian 3-sphere with transverse metric $g$ and positive
transverse scalar curvature $R$.  Then there exists a positive constant $C_0$, which depends only on the first deformation class of $\eta$,
such that for any orbit $l_p$ passing through $p\in S^3$ we have
\begin{equation}\label{volestformula2}
V_{l_p}^M(\frac{\pi}{2\sqrt{R_{\max}}})\geq \frac{C_0}{R_{\max}}.
\end{equation}
\end{proposition}

\begin{proof}
Without loss of generality let $\eta$ be a second deformation of
$$\eta_a=(a_1|z_1|^2+a_2|z_2|^2)^{-1}\eta_0,\quad a_1,a_2\in \mathbb{N}, \quad a_1< a_2, \quad \gcd(a_1,a_2)=1.$$

{\bf Case I, $\mathbb{Z}_{a_2}$-teardrop base space:}
If $a_1=1$, the only singular orbit is
$$l_{z_2}=\{(z_1,z_2)|z_1=0, |z_2|=1\}.$$
Note that
$$L(l_{z_2})=\frac{2\pi}{a_2},$$
and the generic orbits have length $2\pi$.
The base space $\mathcal{B}$ of this foliation, given by
$$\mathcal{B}=S^3/l_{\xi_a}, \quad \pi:S^3\rightarrow \mathcal{B},$$
is an orbifold with an orbifold point
$$ Q=\pi(l_{z_2}).$$
The isotropy group at $Q$ is $\mathbb{Z}_{a_2}$.

Let $l$ be an orbit in $S^3$,
$$l_s=\{x\in S^3| d_{g^M}(x,l)=s\},$$
and
$$T(l, r)=\{x\in S^3| d_{g^M}(x,l)\leq r\}.$$

The injectivity radius of the orbifold point $Q$ in the universal cover of $\mathcal{B}\setminus Q$
is greater than $\pi/\sqrt{R_{\max}}$, see \cite{Wu}. Note that along any geodesic starting from $l_{z_2}$
orthogonally there exists no focal point within distance $\pi/\sqrt{R_{\max}}$.
It follows that

$$\text{minfoc}(l_{z_2})\geq \frac{\pi}{\sqrt{R_{\max}}}.$$
Actually if not, there must be a shortest geodesic $1$-gon at $Q$ of length less than $2\pi/\sqrt{R_{\max}}$
on the universal cover of $\mathcal{B}\setminus Q$, which contradicts with Wu's injectivity radius estimate mentioned above.

Recall that
$$\lambda=K^M_{\max}(\geq 1).$$
By the formula (\ref{Weyl2}), for $s < \text{minfoc}(l_p)$ we have
$$A_{l}^M(s)\geq \pi L(l)\frac{\sin(2s\sqrt{\lambda})}{\sqrt{\lambda}}.$$
Taking $r=\pi/4\sqrt{R_{\max}}$, we get
\begin{eqnarray}\label{teardrop0}V_{l_{z_2}}^M(\frac{\pi}{4\sqrt{R_{\max}}})&=&\int_0^rA_{l_{z_2}}^M(s)d s
\\&\geq& \frac{2\pi^2}{a_2}\int_0^{\frac{\pi}{4\sqrt{R_{\max}}}}\frac{\sin(2s\sqrt{\lambda})}{\sqrt{\lambda}}d s.\nonumber
\end{eqnarray}
Note that from (\ref{Rmax}), we have
$$R_{\max}=2\lambda+6> \lambda,$$
so we get
\begin{equation}\label{teardrop1}
V_{l_{z_2}}^M(\frac{\pi}{4\sqrt{R_{\max}}})\geq
\frac{2\pi^2}{a_2}\int_0^{\frac{\pi}{4\sqrt{R_{\max}}}}\frac{4s}{\pi}ds=\frac{\pi^3}{4a_2}\frac{1}{R_{\max}}.
\end{equation}

For any orbit $l$ such that $$\text{dist}_{g^{\mathcal{B}}}(\pi(l),Q))\leq \frac{\pi}{4\sqrt{R_{\max}}},$$
note that
$$T(l_{z_2},\frac{\pi}{4\sqrt{R_{\max}}})\subset T(l,\frac{\pi}{2\sqrt{R_{\max}}}),$$
by the estimate (\ref{teardrop1}) we get
\begin{equation}\label{teardrop2}
V_l^M(\frac{\pi}{2\sqrt{R_{\max}}})\geq \frac{\pi^3}{4a_2}\frac{1}{R_{\max}},
\quad \forall l : \text{dist}_{g^{\mathcal{B}}}(\pi(l),Q))\leq \frac{\pi}{4\sqrt{R_{\max}}}.
\end{equation}

For any orbit $l$ such that $$\text{dist}_{g^{\mathcal{B}}}(\pi(l),Q)\geq \frac{\pi}{4\sqrt{R_{\max}}},$$
it follows from Wu's injectivity radius estimate \cite{Wu} that there exists a positive constant $C\leq \frac{1}{2}$ such that
$$\text{minfoc}(l)\geq \frac{C}{a_2\sqrt{R_{\max}}}.$$
Hence it follows in a similar way as (\ref{teardrop0}) and (\ref{teardrop1}) that
\begin{equation}\label{teardrop3}V_l^M(\frac{\pi}{4\sqrt{R_{\max}}})\geq \frac{4\pi C^2}{a_2^2}\frac{1}{R_{\max}}, \quad \forall l :
\text{dist}_{g^{\mathcal{B}}}(\pi(l),Q)\geq \frac{\pi}{4\sqrt{R_{\max}}}.
\end{equation}
By (\ref{teardrop2}) and (\ref{teardrop3}), we see that there exists a positive constant $C$ such that
\begin{equation}\label{teardrop4}
V_l^M(\frac{\pi}{2\sqrt{R_{\max}}})\geq \frac{C}{a_2^2}\frac{1}{R_{\max}}, \quad \forall l.
\end{equation}

{\bf Case II, $\mathbb{Z}_{a_1,a_2}$-football base space:}
In case that $a_1\geq 2$, there are just two singular orbits
$$l_{z_2}=\{z_1=0\}, \quad l_{z_1}=\{z_2=0\}.$$
Therefore, the base space of this foliation is an orbifold $\mathcal{B}$ with two orbifold points
$$Q_2=\pi(l_{z_2}), \quad Q_1=\pi(l_{z_1}),$$
here $Q_2$ is of $\mathbb{Z}_{a_2}$ isotropy group  and $Q_1$ is of $\mathbb{Z}_{a_1}$ isotropy group.
Note that
$$L(l_{z_2})=\frac{2\pi}{a_2}, \quad L(l_{z_1})=\frac{2\pi}{a_1},$$
and the generic orbits have length $2\pi$.

The volume estimate (\ref{volestformula2}) follows in this case similarly to the case I,
using Wu's injectivity radius estimates on the base space and the formula (\ref{Weyl2}).
We just outline the main steps.

The injectivity radius of the orbifold point $Q_2$ (respectively $Q_1$) on the universal cover of
$\mathcal{B}\setminus Q_1$ (respectively $\mathcal{B}\setminus Q_2$) is greater than $\pi/\sqrt{R_{\max}}$, which implies that
$$\text{minfoc}(l_{z_i})\geq \frac{\pi}{\sqrt{R_{\max}}}, \quad i=1,2.$$
Hence we have similar volume estimates as (\ref{teardrop2}):
\begin{equation}\label{teardrop5}
V_l^M(\frac{\pi}{2\sqrt{R_{\max}}})\geq \frac{\pi^3}{4a_i}\frac{1}{R_{\max}},\quad \forall l : \text{dist}_{g^{\mathcal{B}}}(\pi(l),Q_i))\leq \frac{\pi}{4\sqrt{R_{\max}}}.
\end{equation}

For any orbit $l$ such that $$\text{dist}_{g^{\mathcal{B}}}(\pi(l),Q_i)\geq \frac{\pi}{4\sqrt{R_{\max}}}, \quad i=1, 2,$$
it follow from Wu's injectivity radius estimate \cite{Wu} that there exists a positive constant $C\leq \frac{1}{2}$ such that
$$\text{minfoc}(l)\geq \frac{C}{a_2\sqrt{R_{\max}}}.$$
Hence we have
\begin{equation}\label{teardrop6}V_l^M(\frac{\pi}{4\sqrt{R_{\max}}})\geq \frac{4\pi C^2}{a_2^2}\frac{1}{R_{\max}},
\end{equation}
for any $l$ satisfies
$$\text{dist}_{g^{\mathcal{B}}}(\pi(l),Q_i)\geq \frac{\pi}{4\sqrt{R_{\max}}}, \quad i=1, 2.$$

By (\ref{teardrop5}) and (\ref{teardrop6}), we see that there exists a positive constant $C$ such that
\begin{equation}\label{teardrop4}
V_l^M(\frac{\pi}{2\sqrt{R_{\max}}})\geq \frac{C}{a_2^2}\frac{1}{R_{\max}}, \quad \forall l.
\end{equation}
\end{proof}

We now handle the irregular case. In this case, the base space is not even an orbifold.
We will use Weyl's tube formula (\ref{Weyl}) to prove the volume estimate (\ref{volest}).
However we shall consider not only a tube about a closed curve but also a tube about a torus.
Here the torus is the closure of some orbit. We first study the geometry of orbit closure.

Let $(S^3,\eta,g^M)$ be an irregular Sasakian $3$-sphere and the contact form $\eta$ is of the first deformation class
\begin{equation}\label{irregulareta}\eta_a=(a_1|z_1|^2+a_2|z_2|^2)^{-1}\eta_0, \quad 1\leq a_1<a_2,\end{equation}
here $a_1/a_2$ is an irrational number and $\eta_0$ is the canonical Sasakian structure given by
$$\eta_0=\sum_{i=1}^2(x^idy^i-y^idx^i).$$
The Reeb vector field determined by $\eta$ is then
$$\xi_a=\sum_{i=1}^2a_i(x^i\frac{\partial}{\partial
y^i}-y^i\frac{\partial}{\partial x^i}),$$
which is generated by the action
$$e^{it}.(z_1,z_2)=(e^{ia_1t}z_1,e^{ia_2t}z_2), \quad t\in [0,+\infty).$$
Let $T_{c_1}$ be the torus in $S^3$ given by
$$T_{c_1}=\{(z_1,z_2)\in S^3: |z_1|^2=c_1^2, |z_2|^2=1-c_1^2\}.$$
For $c_1=0, 1$, the torus $T_{c_1}$ degenerate, respectively, to closed orbits
$$l_{z_2}=\{(z_1,z_2): |z_2|=1\}, \quad l_{z_1}=\{(z_1,z_2): |z_1|=1\}.$$
Note that
$$L(l_{z_i})=\frac{2\pi}{a_i}, \quad i=1, 2.$$
Each orbit is along the torus containing it. In particular for $c_1\in (0,1)$
each orbit is dense in the torus.
Hence the closure of each orbit is a torus and any two points contained in the same torus have transverse distance $0$.
The transverse distance between any two points contained in two different tori equals the (transverse) distance between these two tori.

For $c_1\in (0,1)$, let $P=T_{c_1}$ and $X\in \Gamma(TP)$ denote the unit vector field orthogonal to $\xi_a$ and $N=\Phi X$.
Both vector fields are determined up to choice of directions. For example we can choose
$$N=\frac{\nabla^M |z_1|}{|\nabla^M |z_1||}.$$
The vector field $N$ is the unit normal vector field of the torus,
and $N$ integrates to geodesics in $(S^3,g^M)$.
Note that
$$\nabla^M_{\xi_a}\xi_a=0, \quad \nabla^M_{X}\xi_a=\Phi X=N,$$
so the mean curvature and the second fundamental form of $P$ satisfy the relation
\begin{equation*}
|H|^2-|A|^2=-2.
\end{equation*}

To prove the volume estimate (\ref{volest}) for irregular Sasakian 3-spheres, we encounter similar situation as for quasi-regular ones.
We consider those points close to the ends $l_{z_2}$ and $l_{z_1}$ of the family of tori and those points in the middle part separately.
In view of the formula (\ref{VPM}), we still need to estimate the area of the torus in the middle part.
For this aim, we shall apply the following Lemma.

\begin{lemma} \label{arealowerbound}
Let $(S^3,\eta,\xi_a,\Phi,g^M)$ be a (quasi-regular or irregular)
Sasakian 3-sphere of positive transverse sectional curvature.
Fixing any torus $T_{c_1}$, let
$$(T_{c_1})_t=\{\exp_{\nu(T_{c_1})}tN\},\quad A(t)=\text{Area}[(T_{c_1})_t].$$
Then we have
\begin{equation}\label{arealowerbound1}
A''(t)\leq 0.
\end{equation}
\end{lemma}

\begin{proof}
Note $$A(t)=\int_{T_{c_1}}\Theta_N(p,t)d\mu_P, \quad p\in P=T_{c_1}.$$
Taking $N$ as the formal outward unit normal vector field along the torus $(T_{c_1})_t$
and define the shape operator $S(t)$ by
$$<Su,v>=<\nabla^M_uN,v>, \quad u,v\in T(T_{c_1})_t.$$
Recall the Riccati equation (\ref{Riccati}) and equation (\ref{changeofareaelement}),
i.e.
\begin{equation*} S'(t)=-S(t)^2+R^M_N, \quad S'=\nabla^M_NS, \quad R^M_Nu=R^M(N,u)N,
\end{equation*}
and
\begin{equation*}\Theta_N'(t)=\tr S(t)\Theta_N(t).
\end{equation*}
Note that
\begin{eqnarray*}
(\tr S(t))'&=&N <Se_i,e_i>
\\&=&<\nabla_N^M(Se_i),e_i>+<Se_i,\nabla_N^Me_i>
\\&=&\tr S'+2<\nabla_N^Me_i,Se_i>.
\end{eqnarray*}
Taking $$e_1=\xi_a, \quad e_2=X=-\Phi N,$$
we get
\begin{eqnarray*}
<\nabla_N^Me_i,Se_i>
&=&<\nabla_N^M\xi_a,\nabla^M_{\xi_a}N>+<\nabla_N^MX,\nabla_X^MN>
\\&=&-<X,\nabla^M_{\xi_a}N>+<\nabla_N^MX,\xi_a><\nabla_X^MN,\xi_a>
\\&=&<\nabla^M_{\xi_a}X,N>+<X,\nabla_N^M\xi_a><N,\nabla_X^M\xi_a>
\\&=&<\Phi X,N>+<X,\Phi N><N,\Phi X>
\\&=&0.
\end{eqnarray*}
Therefore we have
$$(\tr S(t))'=\tr(S').$$
Hence we have
\begin{eqnarray*}\Theta_N''(t)&=&[(\tr S(t))^2+\tr S'(t)]\Theta_N(t)
\\&=&[(\tr S(t))^2-\tr S^2(t)-Ric^M(N,N)]\Theta_N(t)
\\&=&[|H|^2-|A|^2-Ric^M(N,N)]\Theta_N(t)
\\&=&[-3-K^M(X,N)]\Theta_N(t)
\\&=&-K(X,N)\Theta_N(t)
\\&\leq&0.
\end{eqnarray*}
\end{proof}

\begin{proposition}\label{prop2}

Let $(S^3,\eta,g)$ be an irregular Sasakian 3-sphere with transverse metric $g$ and positive
transverse scalar curvature $R$.  Then there exists a positive constant $C_0$, which depends only on the first deformation class of $\eta$,
such that for any torus $T_{c_1}$ we have
\begin{equation}\label{volestformula3}
V_{T_{c_1}}^M(\frac{\pi}{2\sqrt{R_{\max}}})\geq \frac{C_0}{R_{\max}}.
\end{equation}
\end{proposition}

\begin{proof}
Let $\eta$ be a second deformation of $\eta_a$ given by (\ref{irregulareta}).
First, note that we may assume that
\begin{equation}\label{distassumption}\text{dist}(l_{z_1},l_{z_2}> \frac{\pi}{2\sqrt{R_{\max}}}.\end{equation}
Otherwise, we have the total volume of $S^3$ and the estimate (\ref{volestformula3}) follow automatically.

By the formula (\ref{Weyl2}), we have
\begin{equation}\label{Weyl3}A_{l_{z_2}}^M(t)\geq \pi L(l_{z_2})\frac{\sin(2t\sqrt{\lambda})}{\sqrt{\lambda}}.\end{equation}
Taking $r=\frac{\pi}{4\sqrt{R_{\max}}}$, we get
\begin{eqnarray*}
V_{l_{z_2}}^M(\frac{\pi}{4\sqrt{R_{\max}}})&=&\int_0^rA_{l_{z_2}}^M(t)dt
\\&\geq&\frac{2\pi^2}{a_2}\int_0^{\frac{\pi}{4\sqrt{R_{\max}}}}\frac{\sin(2t\sqrt{\lambda})}{\sqrt{\lambda}}d t
\\&\geq&\frac{\pi^3}{4a_2}\frac{1}{R_{\max}}.
\end{eqnarray*}
In the same way, we have
\begin{eqnarray*}
V_{l_{z_1}}^M(\frac{\pi}{4\sqrt{R_{\max}}})\geq\frac{\pi^3}{4a_1}\frac{1}{R_{\max}}.
\end{eqnarray*}

For some $i$ and a torus $T_{c_1}$ such that $$\text{dist}(T_{c_1},l_{z_i})\leq \frac{\pi}{4\sqrt{R_{\max}}},$$
note that
$$T(l_{z_i},\frac{\pi}{4\sqrt{R_{\max}}})\subset T(T_{c_1},\frac{\pi}{2\sqrt{R_{\max}}}),$$
so we get
\begin{equation}\label{irregularest1}
V^M_{T_{c_1}}(\frac{\pi}{2\sqrt{R_{\max}}})\geq \frac{\pi^3}{4a_i}\frac{1}{R_{\max}},\quad \forall i, T_{c_1}:
\text{dist}(T_{c_1},l_{z_i})\leq \frac{\pi}{4\sqrt{R_{\max}}}.
\end{equation}

Now for the torus $T_{c_1}$ such that $$\text{dist}(T_{c_1},l_{z_i})\geq \frac{\pi}{8\sqrt{R_{\max}}},\quad i=1, 2,$$
applying Lemma \ref{arealowerbound} and formula (\ref{Weyl3}), we get
\begin{equation}\label{irregular2}
\text{Area}(T_{c_1})\geq \min_i A^M_{l_{z_i}}(\frac{\pi}{8\sqrt{R_{\max}}})\geq \frac{\pi^2}{a_2}\frac{1}{\sqrt{R_{\max}}}.
\end{equation}
For any torus $T_{c_1}$ such that
$$\text{dist}(T_{c_1},l_{z_i})\geq \frac{\pi}{4\sqrt{R_{\max}}},\quad i=1, 2,$$
by the assumption (\ref{distassumption}) and the estimate (\ref{irregular2}), we get
\begin{equation}\label{irregularest2}
V^M_{T_{c_1}}(\frac{\pi}{2\sqrt{R_{\max}}})>V^M_{T_{c_1}}(\frac{\pi}{8\sqrt{R_{\max}}})\geq
\frac{\pi}{4\sqrt{R_{max}}}\frac{\pi^2}{a_2}\frac{1}{\sqrt{R_{\max}}}=\frac{\pi^3}{4a_2}\frac{1}{R_{max}}.
\end{equation}
The proof then follows from (\ref{irregularest1}) and (\ref{irregularest2}).
\end{proof}

\section{Sasaki-Ricci solitons on $S^3$}

In this section we consider existence and uniqueness of gradient Sasaki-Ricci soliton on $S^3$.
In particular we shall prove Theorem \ref{Main2}.
A gradient Sasaki-Ricci soliton $(\eta, -\frac{1}{2}\nabla f)$ on a Sasakian $S^3$ satisfies the equation
\begin{equation}\label{soliton}
\nabla^2f-\frac{1}{2}\triangle_Bfg_\eta=0,
\end{equation}
here $f$ is the basic function defined by $R-r=\triangle_Bf$.

By Belgun's work \cite{Belgun01}, we know that, up to second type deformations,
weighted Sasakian structures are essentially all the Sasakian structures on $S^3$.
Along the Sasaki-Ricci flow, the volume and the total transverse scalar curvature are fixed.
So we can  first compute the average of the transverse scalar curvature, and then consider what second type
deformation of a given weighted Sasakian structure satisfies the gradient Sasaki-Ricci soliton equation.

Let $S^3$ be the unit sphere in $\mathbb{C}^2$. The contact form of
the canonical Sasakian structure on $S^3$ is given by
$$\eta_0=\sum_{i=1}^2(x^idy^i-y^idx^i).$$
The almost complex structure on the canonical Sasakian manifold
$(S^3,\eta_0)$ is induced by the complex structure of
$\mathbb{C}^2$. Moreover the contact distribution $\mathcal{D}$ is
spanned by
$$X_1=-x^2\frac{\partial}{\partial
x^1}+x^1\frac{\partial}{\partial x^2}+y^2\frac{\partial}{\partial
y^1}-y^1\frac{\partial}{\partial y^2},$$
$$X_2=\Phi X_1
=-x^2\frac{\partial}{\partial y^1}+x^1\frac{\partial}{\partial
y^2}-y^2\frac{\partial}{\partial x^1}+y^1\frac{\partial}{\partial
x^2}.$$ Let
$$\sigma=a_1|z_1|^2+a_2|z_2|^2, \quad a_1,a_2>0.$$
The contact form and Reeb vector field of a weighted Sasakian
structure are, respectively, given by
\begin{equation} \eta_a=\sigma^{-1}\eta_0, \quad
\xi_a=\sum_{i=1}^2a_i(x^i\frac{\partial}{\partial
y^i}-y^i\frac{\partial}{\partial x^i}).
\end{equation}
Now let
$$Z_1=|z_2|^2(y^1\frac{\partial}{\partial x^1}-x^1\frac{\partial}{\partial y^1})
-|z_1|^2(y^2\frac{\partial}{\partial
x^2}-x^2\frac{\partial}{\partial y^2}),$$
$$Z_2=\Phi Z_1=|z_2|^2(x^1\frac{\partial}{\partial x^1}+y^1\frac{\partial}{\partial y^1})
-|z_1|^2(x^2\frac{\partial}{\partial
x^2}+y^2\frac{\partial}{\partial y^2})$$
and
$$Z=\sigma^{-1} (Z_1-iZ_2)=-2i\sigma^{-1}|z_1|^2|z_2|^2(\frac{1}{\overline{z}_1}\frac{\partial}{\partial z_1}
-\frac{1}{\overline{z}_2}\frac{\partial}{\partial z_2}).$$
Here
$$Z_i\in \ker\eta_a,\quad [\xi_a,Z_i]=0, \quad i=1,2,$$
and
$$\Phi Z=iZ, \quad [\xi_a,Z]=0.$$
In particular on $(S^3,\eta_a)$, $Z_1$ is along the torus
$T_{c_1}:=\{|z_1|^2=c_1^2,|z_2|^2=1-c_1^2\}$, and $Z_2$ is
perpendicular to $Z_1$. Let $g_a$ be the transverse metric
associated to $\eta_a$ and $\widetilde{g}_a=g_a(Z,\overline{Z})$. Then we have
$$\widetilde{g}_a=g_a(Z,\overline{Z})=2\sigma^{-3} |z_1|^2|z_2|^2, \quad [Z,\overline{Z}]=-2i\widetilde{g}_a\xi_a.$$
So on $(S^3,\eta_a)$ we have
$\nabla_Z\overline{Z}=\nabla_{\overline{Z}}Z=0.$

\begin{proposition}
The weighted Sasakian manifold $(S^3, \eta_a)$ has
\begin{eqnarray*}
R(g_a)=-24(a_1-a_2)^2\sigma^{-1}|z_1|^2|z_2|^2-16(a_1-a_2)(|z_1|^2-|z_2|^2)+8\sigma,
\end{eqnarray*}
$$r=4(a_1+a_2).$$
\end{proposition}
\begin{proof}
The transverse scalar curvature can be computed by
$$R(g_a)=-2\widetilde{g}_a^{-1}Z\overline{Z}\log (\widetilde{g}_a).$$ Then we have
\begin{eqnarray*}
\int_{S^3}R\eta_a\wedge d\eta_a
&=&\int_{|z_1|^2=t}R \sigma^{-2}(2\pi\sqrt{t})(2\pi\sqrt{1-t})\frac{dt}{|\nabla t|_{g^M_{\eta_0}}}
\\&=&\int_0^1R \sigma^{-2}2\pi^2dt
\\&=&8\pi^2\frac{a_1+a_2}{a_1 a_2},
\end{eqnarray*}
and
$$\int_{S^3}\eta_a\wedge d\eta_a=\int_0^1\sigma^{-2}2\pi^2dt=\frac{2\pi^2}{a_1a_2}.$$
\end{proof}

\begin{theorem}
On each weighted Sasakian manifold $(S^3,\eta_a)$,
there exists a unique Sasakian structure $\eta$ homologous to $\eta_a$ so that it is a gradient Sasaki-Ricci soliton.
\end{theorem}
\begin{proof}
The existence and uniqueness of Ricci soliton on bad orbifolds were obtained by Wu \cite{Wu}.
When $(S^3,\eta_a)$ is a quasi-regular Sasakian structure, the problem is reduced to Wu's result.
So we focus on the irregular case.

Assume
$$\eta=\eta_a+i(\overline{\partial}_B-\partial_B)\varphi$$
is the contact form of a gradient Sasaki-Ricci soliton on $(S^3,\eta_a)$ and $g_\eta$ its transverse metric.
Here $\varphi$ is a basic function. So we have
$$\widetilde{g}:=d\eta(Z,\Phi\overline{Z})=2\sigma^{-3}
|z_1|^2|z_2|^2+\varphi_{Z\overline{Z}}, \quad
R(g_\eta)=-2\widetilde{g}^{-1}Z\overline{Z}\log \widetilde{g}.$$
Let $f$ be the basic function defined by
$R(g_\eta)-r(g_\eta)=\triangle_Bf.$
By the definition of a gradient Sasaki-Ricci soliton, we have
\begin{equation}\label{soliton}
\nabla^2f-\frac{1}{2}\triangle_Bfg_\eta=0.
\end{equation}
It's equivalent to the system:
\begin{equation}\label{soliton0}
ZZf-Z(\log \widetilde{g})Zf=0,
\end{equation}
and
\begin{equation}\label{soliton1}
Z\overline{Z}f=\frac{1}{2}(R(g_\eta)-r)\widetilde{g}.
\end{equation}

For an irregular Sasakian structure, all generic orbits of the
characteristic foliation are  dense in the torus $T_{c_1}$ containing it.
Hence basic function take constant value on each torus.
Therefore, basic functions $f$ and $\widetilde{g}$ have vanishing derivatives in the direction of $Z_1$.
Note also that
$$[Z_1,Z_2]=-2\sigma^{-1}|z_1|^2|z_2|^2(\xi_a+Z_1).$$

By equation (\ref{soliton0}), we have
$$ -\sigma^{-1}Z_2(\sigma^{-1}Z_2f)+\sigma^{-1}Z_2(\log \widetilde{g})(\sigma^{-1}Z_2f)=0.$$
So we get
\begin{equation}\label{soliton00}
Z_2f=-c\sigma \widetilde{g}.
\end{equation}
Let $\pi_\eta$ be the projection $\pi_\eta: TS^3\rightarrow \ker\eta$ and $\widetilde{Z}_i=\pi_\eta Z_i$.
The vector fields $\widetilde{Z}_i$ satisfy
$$g_\eta(\widetilde{Z}_i,\widetilde{Z}_j)=d\eta(Z_1,Z_2)\delta_{ij}=\frac{1}{2}\sigma^2\widetilde{g}\delta_{ij}.$$
On the other hand,
$$g_\eta(\nabla f,\widetilde{Z}_2)=Z_2f=-c\sigma \widetilde{g},$$
so we see that
\begin{equation}\label{vector}
\nabla f=-2c\sigma^{-1}\widetilde{Z}_2.
\end{equation}
Let $X=\sigma^{-1}Z_2$, then by
(\ref{soliton00}) and the expression of $R(g_\eta)$, equation (\ref{soliton1}) can be written as
\begin{equation}
X^2\log \widetilde{g}-cX(\widetilde{g})+2(a_1+a_2)\widetilde{g}=0.
\end{equation}
Without loss of generality, we assume $a_1<a_2$. Let
$\kappa=2(a_1+a_2)$ and
$$s(\sigma)=-\frac{a_1}{2}\log (\sigma-a_1)+\frac{a_2}{2}\log (a_2-\sigma)\in (-\infty, +\infty).$$
Then we have $X(s)=1$.
Write $\widetilde{g}=\widetilde{g}(s)$, $\widetilde{g}'=\frac{d\widetilde{g}}{ds}$ and
$\widetilde{g}''=\frac{d^2\widetilde{g}}{ds^2}$.
Therefore
\begin{equation}\label{ode1}
(\frac{\widetilde{g}'}{\widetilde{g}})'-c\widetilde{g}'+\kappa \widetilde{g}=0.
\end{equation}
That's the same equation appeared in \cite{H}. Let $\widetilde{g}=v'$.
The function $v$ is determined up to a constant. We include the constant in $v$.
Then equation (\ref{ode1}) integrates to
\begin{equation}
v''-cv'^2+\kappa vv'=0.
\end{equation}
Integrating the last equation again, we get
\begin{equation}
v'=\frac{\kappa}{c}v+\frac{\kappa}{c^2}(1-ke^{cv}),
\end{equation}
here $k$ is a constant. Let $y=cv+1, u=\kappa s/c$. We get
\begin{equation}
\frac{dy}{du}=y-ke^{y-1}.
\end{equation}
Then \begin{equation}\label{yandu}
u=\int\frac{dy}{y-ke^{y-1}}.
\end{equation}
So we have
\begin{equation}
\widetilde{g}=\frac{y'}{c}, \quad \int\frac{dy}{y-ke^{y-1}}=\kappa s/c.
\end{equation}
The denominator in (\ref{yandu}) must have two zeroes, which happens
precisely as $0<k<1$. Now let $0<k<1$, $y_1=1-p<1$ and $y_2=1+q>1$
be two solutions to $y=ke^{y-1}$.
Moreover as $y\in(y_1,y_2)$ goes to $y_1$ and $y_2$,
$u$ tends to minus infinity and positive infinity respectively.
Note that $s$ goes to negative infinity
and positive infinity as a point goes to the points $z_1=0$ and $z_2=0$ respectively.

We now consider the asymptotic behavior near the points $z_1=0$ and
$z_2=0$.
Suppose as $s\rightarrow -\infty$, we have an expansion
$$\widetilde{g}(s)=b_1e^{\lambda s}+b_2e^{2\lambda s}+\cdots $$
in powers of $e^{\lambda s}$.
Similarly, as $s\rightarrow +\infty$, suppose we have an expansion
$$\widetilde{g}(s)=d_1e^{-\mu s}+d_2e^{-2\mu s}+\cdots $$
Near $y_1$ write $y=y_1+h$. Then we have
$$y-ke^{y-1}=ph-\frac{1}{2}(1-p)h^2+\cdots $$
So we have $u=\frac{1}{p}\log h+\cdots $, which implies that
$$dy=dh=pe^{pu}du+\cdots $$
Hence
$$\widetilde{g}(s)=v'=\frac{dv}{ds}=\frac{\kappa}{c^2}\frac{dy}{du}=\frac{\kappa p}{c^2}e^{pu}+\cdots $$
So we have $\lambda=\frac{\kappa p}{c}$ if $c>0$, or
$\mu=-\frac{\kappa p}{c}$ if $c<0$. In a similar way, we have
$\mu=\frac{\kappa q}{c}$ if $c>0$, or $\lambda=-\frac{\kappa q}{c}$ if $c<0$.

Write $\varphi$ as a function of $\sigma\in [a_1,a_2]$, then by the expression of
$$\widetilde{g}=2\sigma^{-3}|z_1|^2|z_2|^2+\varphi_{Z\overline{Z}},$$
we see that near the points with $z_2=0$ either $\widetilde{g}$ is of order of $|z_2|^2$ or of order of $|z_2|^4$.
However, by the nondegeneracy of $g_\eta$, it can only be of order of $|z_2|^2$. On the other hand
$$(a_2-a_1)|z_2|^2=e^{-\frac{2}{a_1}s}[(a_2-a_1)|z_1|^2]^{\frac{a_2}{a_1}},$$
so we have $$\mu=\frac{2}{a_1}.$$ Similarly near the points with $z_1=0$,
$\widetilde{g}$ is of order of $|z_1|^2$, and in the same way we have
$$\lambda=\frac{2}{a_2}.$$
We always have $p<q$, see \cite{H}. Hence $c>0$ and
$\frac{\lambda}{\mu}=\frac{a_1}{a_2}=\frac{p}{q}$.

Therefore we see that to get a gradient Sasaki-Ricci soliton solution on an irregular Sasakian $3$-sphere $(S^3,\eta_a)$
is to find a constant $k\in (0,1)$ such that the equation $y=ke^{y-1}$ has two solutions
$$y_1=1-p<1, \quad y_2=1+q>1,$$
and
$$0<\frac{a_1}{a_2}=\frac{p}{q}<1.$$
The existence and uniqueness then follow exactly as \cite{Wu}.
\end{proof}


\begin{thebibliography}{AB}


{\small \setlength{\parskip}{10pt} \setlength{\baselineskip}{7pt}}



\bibitem{Belgun00}
Belgun, F. A., On the metric structure of non-K\"ahler
complex surfaces. Math. Ann. 317 (2000), no. 1, 1--40.

\bibitem{Belgun01}
Belgun, F. A., Normal CR structures on compact
3-manifolds. Math. Z. 238 (2001), no. 3, 441--460.

\bibitem{BG1}
Boyer, C. P.; Galicki, K., On Sasakian-Einstein geometry.
Internat. J. Math. {\bf 11} (2000), 873--909.

\bibitem{BG2}
Boyer, C. P.; Galicki, K.,New Einstein metrics in dimension five, {\it
J. Differential Geom. } {\bf 57}  (2001), 443--463.

\bibitem{BGK}
Boyer, C. P.; Galicki, K.; Koll\'or, J.,  Einstein metrics on spheres,  {\it Ann. of Math.,}
 (2) {\bf  162}  (2005), 557--580.

\bibitem{BGM}
Boyer, C. P.; Galicki, K.; Matzeu, P., {\em On Eta-Einstein Sasakian
  Geometry}, Comm. Math. Phys., 262 (2006), pp. 177--208.

\bibitem{BGN}
Boyer, C. P.; Galicki, K.; Nakamaye, M.,
 On the geometry of Sasakian-Einstein 5-manifolds,
 {\it   Math. Ann.,}  {\bf  325}  (2003), 485--524.

%Boyer, Charles ; Galicki, Krazysztof
\bibitem{Boyer}
Boyer, C.; Galicki, K., Sasakian geometry,  New York, NY : Oxford Univ. Press, 2008.
%Boyer, Charles P.; Galicki, Krzysztof; Matzeu, Paola. On
%eta-Einstein Sasakian geometry. Comm. Math. Phys. 262 (2006), no. 1,
%177--208.

\bibitem{Cao} Cao, Huai Dong,
Deformation of K\"ahler metrics to K\"ahler-Einstein metrics on
compact K\"ahler manifolds,{\it  Invent. Math.,} {\bf  81} (1985) 359--372.

\bibitem{Chow}
Chow, Bennett. The Ricci flow on the $2$-sphere. J. Differential
Geom. 33 (1991), no. 2, 325--334.

\bibitem{ChowWu}
Chow, Bennett; Wu, Lang-Fang. The Ricci flow on compact
$2$-orbifolds with curvature negative somewhere.  Comm. Pure Appl.
Math. 44 (1991), no. 3, 275--286.

\bibitem{Wang} Futaki, A.; Ono, H.; Wang, G., Transverse K\"ahler geometry of Sasaki manifolds and toric Sasaki-Einstein
manifolds,  J. Differential Geom.  {\bf 83 } (2009),  no. 3, 585--635.


\bibitem{GO}
Gauduchon, P.; Ornea, L. Locally conformally K\"ahler metrics on
Hopf surfaces. Ann. Inst. Fourier (Grenoble) 48 (1998), no. 4,
1107--1127.

\bibitem{GMSW}
Gauntlett, J.P.; Martelli, D.; Sparks, J.; Waldram, W., {
Sasaki-Einstein metrics on $S^2\times S^3$}, {\it  Adv. Theor. Math.
Phys.,}  {\bf 8} (2004),  711--734

\bibitem{GMSW2}
Gauntlett, J.P.; Martelli, D.; Sparks, J.; Waldram, W.,
 {\it A new infinite class of Sasaki-Einstein manifolds}, 2004,
 Adv. Theor. Math. Phys., {\bf 8} (2004), 987--1000

\bibitem{Geiges}
Geiges, H., Normal contact structures on $3$-manifolds. Tohoku
Math. J. (2) 49 (1997), no. 3, 415--422.

\bibitem{Gray}
Gray, A., Tubes. Second edition. Progress in Mathematics, 221. Birkh\"auser Verlag, Basel, 2004. xiv+280 pp.

%\bibitem{Sparks1} J. P. Gauntlett, D.  Martelli, J. Sparks, and D. Waldram, Sasaki-Einstein Metrics on $S^2\times S^3$
%Adv. Theor. Math. Phys. Volume 8, Number 4 (2004), 711-734.


\bibitem{H}
Hamilton, R. S., The Ricci flow on surfaces. Mathematics and
general relativity (Santa Cruz, CA, 1986), 237--262, Contemp. Math.,
71, Amer. Math. Soc., Providence, RI, 1988.

\bibitem{Oo}
Lovri\'c, Miroslav; Min-Oo, Maung; Ruh, Ernst A. Deforming
transverse Riemannian metrics of foliations. Asian J. Math. 4
(2000), no. 2, 303--314.


\bibitem{MSY} Martelli, D.; Sparks, J.;  Yau,  S.-T.,
	\emph{Sasaki-Einstein manifolds and volume minimisation}, Comm.
	Math. Phys. {\bf 280} (2008),  no.3, 611--673.

\bibitem{SWZ}
K. Smoczyk; G. Wang; Y.B. Zhang. The Sasaki-Ricci flow,  Internat. J. Math.  {\bf 21 } (2010),  no. 7, 951--969.


\bibitem{Wu} Wu, Lang-fang.  The Ricci flow on 2-orbifolds with positive curvature. J. Differential Geom. 33 (1991), no. 2, 575--596.

\bibitem{Ye} Ye, R. Entropy Functionals, Sobolev Inequalities and kappa-Noncollapsing Estimates along the Ricci Flow. arXiv:0709.2724v1.



\end{thebibliography}
\end{document}